\documentclass[13pt]{article}
\usepackage{amsmath}
\usepackage{amssymb}
\usepackage{amsthm}
\usepackage{paralist}
\usepackage[all]{xy}
\usepackage{graphics,psfrag,graphicx}
	\usepackage[active]{srcltx}

\newtheorem{theorem}{Theorem}[section]
\newtheorem{lemma}[theorem]{Lemma}
\newtheorem{corollary}[theorem]{Corollary}
\newtheorem{proposition}[theorem]{Proposition}

\theoremstyle{definition}
\newtheorem{remark}[theorem]{Remark}
\newtheorem{definition}[theorem]{Definition}
\newtheorem{examples}[theorem]{Examples}
\newtheorem{example}[theorem]{Example}

\def\R{\mathbb{R}}
\def\F{\mathcal{F}}
\def\U{\mathcal{U}}
\def\L{\mathcal{L}}

\def\C{\mathbb{C}}    
\def\s{\mathbb{S}}   
\def\B{\mathbb{B}}    
\def\H{\mathcal H}     
\def\e{\varepsilon}       
\def\l{\lambda}
\def\0{\underline 0}

\def\n{\noindent}
\def\nn{\vskip.2cm \noindent}
\def\uz{{z}}
\def\oz{\overline{z}}
\def\nz{\mathbf{z}}
\def\n0{\mathbf{0}}
\def\ny{\mathbf{y}}
\def\np{\mathbf{p}}
\def\nv{\mathbf{v}}
\def\nw{\mathbf{w}}

\author{Beatriz Limon Gutierrez and Jose Seade}

\title{Morse Theory and the topology of holomorphic foliations near an isolated singularity\footnote{Research partially supported by PAPIIT-UNAM, CONACYT (Mexico) grants 55084 and 162340, and by  CNRS-CONACYT  through the Laboratorio International  Solomon Lefschetz (LAISLA).   \newline $\quad$ {\it
Key-words:} Holomorphic foliation, Morse function, Morse index, polar variety, contact point, gradient flow.
\newline   {\it Mathematics Subject Classification.} Primary: 32S65, 37F75  
  Secondary:  53C12, 58F18 }}

\author{ Beatriz Lim\'on and Jos\'e Seade  \\[12pt]
Instituto de Matem\'aticas, Unidad Cuernavaca\\
Universidad Nacional Aut\'onoma de M\'exico}

\begin{document}
\date{}
\maketitle
\vskip0,2cm\noindent
\setcounter{section}{0}

\begin{abstract}
Let $\F$ be the germ at $\n0 \in \C^n$ of a holomorphic foliation of dimension $d$, $1 \le d < n$,  with an isolated singularity at $\n0$.
We study its geometry and topology using ideas that originate in the work of Thom concerning Morse theory for foliated manifolds. Given $\F$ and a real analytic  function $g$   on $\C^n$ with a Morse critical point of index 0 at $\n0$, we look at the corresponding polar variety $M= M(\F,g)$. These are  the  points of contact of the two foliations, where $\F$ is tangent to the fibres of $g$.    This is analogous to the usual theory of polar varieties in algebraic geometry, where  holomorphic functions are studied by looking at the intersection of their fibers  with those of  a linear form. Here we replace the linear form by a real quadratic map, the Morse function $g$.
  We then study  $\F$ by looking at the intersection of its leaves with the level sets of $g$, and the way how these intersections change as the sphere gets smaller. 
\end{abstract}

\section*{Introduction}
In 1964 Ren\'e Thom wrote a beautiful article \cite{Thom} explaining how the classical ideas of Morse theory  can be adapted to studying the topology of non-singular foliations on smooth manifolds.  
The purpose of this work is to adapt those ideas  to studying the topology of holomorphic foliations near an isolated singular point. 
      
  We consider a holomorphic foliation $\F$ of dimension $d \ge 1$ defined on an open neighbourhood $\U$ of the origin $\n0 \in \C^n$, with a unique singular point at $\n0$. To study its topology, we consider a  function $g$ on $\U$ with a Morse critical point at $\n0$ of index $0$, so that its non-critical levels  are diffeomorphic to spheres; call these simply ``spheres". Look  at the restriction  of $g$ to the leaves of $\F$. The critical points of  $g|_\L$ are the points of contact of $\F$ with the fibres of $g$,  where the leaves of $\F$ are tangent to the fibres of $g$.
 The set of all such points is the {\it polar variety} of $\F$ relative to $g$, in 
   analogy with this classical notion  in algebraic geometry, where the role of $g$ is usually played by linear forms.

  Inspired by the way how J. Milnor studied the topology of complex singularities in  \cite{Milnor:singular}, we want to study the topology of $\F$ near $\n0$ by looking at the intersection of the leaves of $\F$ with a sufficiently small sphere $\s_\e$, and the way how these intersections change as we make the radius of the sphere tend to $0$. 
  In fact this  point of view
  for studying holomorphic foliations already appears in  \cite{Arnold, CKP, GSV}, and this work is very much indebted those articles, specially to \cite{GSV}.

Of course the contacts of 
 the leaves of $\F$ with the fibres of $g$ can be degenerate. If the contacts are all non-degenerate, then we say that $\F$ {\it carries a Morse structure compatible with} $g$, or simply that $M$ carries the Morse structure of $g$, for short.
In the sequel we give various examples of both, degenerate and non-degenerate contacts.
   
   Our first result (Theorem \ref{Theorem 1})  is the analogous in our setting of a classical result for polar varieties 
(see for instance \cite{Le}):  we prove that the restriction of $g$ to the leaves of $\F$ is a Morse function on each leaf if and only if away from $\n0$, 
   the polar variety  $ M^*= M\setminus \{\n0\}$ is a smooth, reduced, 
 submanifold of $\mathbb{C}^n$ of codimension $2d$ and $M^*$ is everywhere transversal to $\mathcal{F}$. 
   
   The topology of holomorphic foliations near a singular point can be rather complicated. 
   In order to have a certain understanding of the behaviour of $\F$ near $\n0$ we ask for the condition that the polar set $M = M ( \F , g)$ be real analytic at $\n0$. This happens in many 
    interesting families, including all foliations of dimension or codimension one. In this case we say that the foliation is {\it contact-analytic} with 
respect to the Morse function $g$. We  have (Theorem \ref{Theorem 2}) that if 
$\mathcal{F}$ is   contact-analytic and   carries the Morse structure of $g$, then the corresponding polar variety has finitely many 
 irreducible components, say $M_1,...,M_r$, all of them pairwise disjoint away from $\n0$; each of these is smooth away from $\n0$, of real codimension $2d$, transversal to the foliation and consists of points where the contacts have all the same Morse index.

   Using this, we can give the following topological picture of $\F$ near $\n0$ (Theorem \ref{t: Morse theory}):  we equip $\F$  with {\it the gradient flow on the leaves}, {\it i.e.}, the flow
   $\mathcal G_\F$     of the    vector field on $\U$ obtained  by projecting  at each point the gradient of $g$   to the tangent space of $\F$. Let $\B_\e$ be a sufficiently small $g$-ball  centred at $\n0$. This ball splits in two disjoint
   $\mathcal G_\F$-invariant sets: the saturated $\widehat M$ of $M \cap \B_\e$ by $\F$ and its complement
   $K := \B_\e \setminus \widehat M$. 
   On $K$ the topology is somehow simple and  the dynamics can be  rich, while on $\widehat M$ the topology is rich and the dynamics is often simpler.

      Each leaf $\L$  in $K$ is homeomorphic to  a product $(\L \cap \s_\e) \times \R$, immersed in $\C^n$ so that it is transversal to each $g$-sphere around $\n0$ and the $\alpha$-limit of  each orbit in $K$ of the gradient flow $\mathcal G_\F$ is the origin $\n0$. Thus, for each positive number $\e' < \e$, the leaves of $\F$ in $K$ meet the $g$-sphere 
   $\s_{\e'}$ transversally and define a real analytic foliation on it, which can have very rich dynamics. For instance \cite{CKP}, if $\F$ is defined by a linear vector field in the Poincar\'e domain with generic eigenvalues, then the foliation on $\s_{\e'}$ actually is defined by a flow which is Morse-Smale. 
   
   The picture on $ \widehat M$ is rather different: the  {\it $\alpha$-limit} of $\mathcal G_{\F}$ 
 of each  leaf in
$\widehat M$  is the set of points where the corresponding leaf meets the polar variety $M$. The intersection of the leaves of $\F$ with $M$ is always transverse, as described above, and 
     the topology of every leaf $\L \subset \widehat M$  is determined by its intersection with the boundary sphere, $\L \cap \s_\e$, and the points where
$\L$ meets $M$: each such intersection point comes with a Morse index, that tells us what type of handle we must attach to the leaf when passing through that point.
 In particular, if $\L$ is compact, then its Euler-Poincar\'e characteristic $\chi(\L)$ equals the number of intersection points in $\L \cap M$, counted with sign. The sign is negative when the corresponding Morse index is odd, and 
positive   otherwise.
   
   For instance, for generic linear actions of $\C^m$ in $\C^n$ in the Siegel domain (see \cite{Me} and Example \ref{ex: siegel} below), $m << n$, each Siegel leaf $\L$ is a copy of $\C^m$ embedded in $\C^n$ with a unique point of minimal distance to $\n0$. The gradient flow of the quadratic form $(z_1,\cdots,z_n) \mapsto \vert z_1\vert^2 + \cdots + \vert z_n\vert^2$ describes each Siegel leaf as a cylinder over its intersection with a sphere, $\L \cap \s^{2n-1} \cong \s^{2m-1} $, to which we attach a handle of Morse index 0, {\it i.e.}, a $2m$-disc.

  We remark that the topology of the polar varieties that arise in this way can be rather interesting. In fact these are all  manifolds with a canonical complex structure determined by a foliated atlas for $\F$ (see Remark \ref{r: Haefliger}), and with a rich geometry. 
  For instance, when $\F$ is the foliation of a linear flow in the Siegel domain with generic eigenvalues and $g$ is the usual metric, the manifold $M_\e = M \cap \s_\e$ is the space of Siegel leaves of the flow and it has very interesting topology \cite{CKP, LM}.  This type of manifolds has been studied by several authors giving rise to the theory of {\it LV-M manifolds} (see for instance \cite {LM-V, Me, MV}). These  belong also to an interesting class of manifolds that appear in mathematical physics and algebraic topology,  called {\it moment-angle manifolds} (see for instance \cite{GL}). On the other hand, the polar varieties associated to non-linear vector fields in $\C^n$ of the form $(\lambda_1 z_{\sigma(1)}^{a_1}, \cdots, \lambda_n z_{\sigma(n)}^{a_n})$, where the $\lambda_i$ are non-zero complex numbers, $a_i \ge 2$, and $\sigma$ is a permutation of the set $\{1,\cdots,n\}$, have been studied in \cite{RSV, Seade, Seade2} in relation with Milnor fibrations for real singularities. In all these cases the function that defines the polar variety determines a Milnor open-book \cite{Milnor:singular}. The study of this type of singularities has evolved and several related lines of research are being pursued by various authors, see for instance 
    \cite{BPS, Cis, CSS2, Oka2}.

 Notice that if $\s_\e $ is a small $g$-sphere in $\U$ around $\n0$, then 
 the intersection of the leaves of $\F$ with
$ \s_{\e}$ defines a real analytic foliation $\F_{\e}$ on the
sphere, which is singular at $M_\e  := \s_{\e} \cap M $, where $M$ is the polar variety.
 Away from $M_\e$ the leaves of $\F_{\e}$ have
real dimension $2d-1$.  In Section \ref{s: spheres}   we address the problem of studying how the topology of the foliation $\F_\e$ varies as $\e \rightarrow 0$. For instance if $M = \{\n0\}$, as it happens for linear vector fields in the Poincar\'e domain, then 
the topology of $\F_{\e}$ is independent of $\e$. This happens also for generic  flows in the Siegel domain, and more generally for generic linear $\C^m$-actions on $\C^n$ in the Siegel domain.
In \cite{GSV} it was proved that if $\F$ has dimension 1 and all the points in $M$ have Morse index 0, {\it i.e.}, they correspond to local minimal points in their leaves, then the topology of $\F_{\e}$ is independent of $\e$. Here we show that the same statement is true for foliations of arbitrary dimension (Theorem \ref  {Theorem 3}).

Finally,  Section \ref{s: codim 1} is motivated by \cite {Ito-Scardua}, where  the authors prove that
every linear foliation  $\F$ in $\C^n$ has  associated an invertible symmetric  $n \times n$ complex matrix, 
and   the set of all such matrices that define a foliation that carries the Morse structure of 
$Q= |z_1|^2 +   |z_2|^2 + \cdots + |z_n|^2 $ is  dense. 
In fact their complement is a  Zariski closed proper subset.
Here  we look at homogeneous foliations of degree $k >1$ and show that up to isotopy,  all these carry the Morse structure of the usual hermitian metric in $\C^n$.

\section{Some definitions and examples}\label{sec. definitions}

Recall that if $M$ is a smooth (= $C^\infty$) manifold and $g: M \to \R$ is a smooth function with isolated
 critical values, then $g$ is a Morse function if its critical points are all non-degenerate. Usually it is also asked that  different 
critical points correspond to distinct critical values; yet, in this article we ignore that condition.
 The {\it Morse lemma} says that in a neighbourhood of a
 critical point, one can choose a  coordinate chart such that  $g$ takes locally the
 form:
$$g(x_1,...,x_k,....,x_m) = - x_1^2 - ... -x_k^2 + x_{k+1}^2 +.... + x_m^2\,,$$
where $m$ is the dimension of $M$. The number $k$ of negative signs is called {\it the Morse index} 
of $g$ at  the critical point. Points of Morse index $0$ or $m$ correspond to local minimal and maximal  
points respectively. All others are saddle-points.

In this work  $M$ is a  neighbourhood $\U$ 
of the origin $\n0$ in $\C^n$, $n \ge
2$, and $g: \U \to \R$ is a Morse function with a unique critical point at $\n0$ of Morse
 index $0$. Its level surfaces near $\n0$ are diffeomorphic to $(2n-1)$-spheres.

The following concept  is introduced in  \cite{GSV} for 1-dimensional foliations under the name of {\it foliations of Morse-type}. See also \cite{RSV} where these are studied in relation with finite determinacy of map-germs, and
 also \cite{Ito-Scardua, Ito-Scardua2} where the concept is studied for certain codimension 1 foliations. We find that  the name ``Morse-type foliation" can be confused with ``Morse foliation" (these are real codimension 1 foliations which are locally defined by a Morse function). Thus we use a different name here.
 
Unless  stated otherwise, throughout this article $\F$ is  a holomorphic foliation   of dimension $d \ge 1$ defined
on  $\U$,  with a unique singular point at $\n0$, and $g$ is a Morse function on $\U$ with a unique critical point at $\n0$ of
 Morse index $0$. (See \cite{BB, Su} for basic material on holomorphic foliations.) 
A leaf of $\F$ means a leaf on  $U \setminus \{\n0\}$.

\begin{definition}\label{foliacion-morse}  We say that $\mathcal{F}$  \emph{has a Morse structure compatible with that of} $g$ if  the restriction $g_\L$ of $g$ to each leaf $\mathcal{L}$
of $\mathcal{F}$, is a Morse function on $\mathcal{L}$. For the sake of  ``compactness", in the above situation we will often say that $\F$ {\it carries the Morse structure} of $g$.
\end{definition}

\smallskip

 The   lemma below follows easily from \cite{Thom} and its proof is left to the reader:
\begin{lemma}\label{puntos-crit-tangentes}
 The critical points of $g_{\mathcal{L}}$ are the points where $\mathcal{L}$ is tangent to the level surfaces
of $g$.
\end{lemma}

Denote by $M = M(\F,g)$ the set of points in
$\U$ where the two foliations are tangent. Lemma \ref{puntos-crit-tangentes} says that these are exactly the critical points of the restrictions   of $g$ to the leaves of $\F$.
These are the points of
\emph{ contact} of the two foliations in the notation of Thom \cite {Thom}. We assume, by
definition, that this set contains $\n0$, since at this point both
foliations are singular and their tangent spaces are
$0$-dimensional.

\begin{definition}\label{contactos}
We call  $M := M(\mathcal{F},g)$ the \emph{polar set } of $\mathcal{F}$ \emph{with respect to the Morse function} $g$. 
\end{definition}

For each point  $\nz \in \U \setminus \{\n0\}$, let $\L_\nz$ be the leaf of $\F$ that contains $\nz$ and  $T_\nz\L_\nz$ its tangent space.  Then  $M^* := M \setminus \{\n0\}$ is:
$$M^* := \{ \nz \in \U \subset \C^n \, \big | \, {\rm Re} \langle v, \nabla g(\nz) \rangle = 0 \;, \quad \forall v \in T_\nz\L_\nz\}\,,$$
where $\nabla \,g$ is the gradient vector field and $\langle\,,\, \rangle$ is the usual hermitian product, whose real part is the usual inner product in 
$\R^{2n} \cong \C^n$.

Notice that at each point $\nz \in \L$ we have a  vector $\nabla g_\L(\nz) \in T_\nz \L _\nz$, which is the projection to  $T_\nz \L _\nz$  of the gradient vector  $\nabla g(\nz)$. In fact $\nabla g_\L(\nz) \in T_\nz \L _\nz$ is the gradient of 
 the restriction of  $g$ to the leaf $\L_\nz$. We thus have a smooth vector field $\nabla g_\F$ on $\U$, which is everywhere tangent to 
 the leaves of $\F$. 
 
 Let   ${\mathcal G_\F}$ be the (local) flow of  $\nabla g_\F$.  We call this the {\it gradient flow on the leaves of $\F$ for the   function $g$}, in analogy with \cite{GSV}. 
 This flow 
leaves invariant each  leaf of $\F$, and at each point moves along its  leaf in the direction that (travelling backwards) brings 
it ``closer" to $\n0$ with respect to the level sets of the Morse function $g$.

 The fixed points of ${\mathcal G_\F}$, others than $\n0$, are the points where $\nabla g_\L(\nz)$ is orthogonal to   $T_\nz \L_\nz$. These are exactly the points of contact of $\F$ with the fibres of $g$. Hence 
 we have:
 
 \begin{lemma} The set $M$ of points where $\F$ is tangent to the level hypersurfaces of $g$ is the  singular set (fixed points)  of the gradient flow ${\mathcal G_\F}$.  Moreover, $\F$ carries the Morse structure of $g$ if and only if the singularities of ${\mathcal G_\F}$ on each leaf are isolated and non-degenerate.
 \end{lemma}

\medskip

\begin{examples}\label{ex: siegel}
In \cite {CKP} the authors
study the topology of  
flows defined by  a linear, diagonal,  vector field $F$ in $\C^n$ with
eigenvalues $\Lambda = \{\l_1,...,\l_n\}$. From the viewpoint of
dynamics, maybe the most interesting case is when $F$ is in the
{\it Poincar\'e domain}, {\it i.e.}, the convex hull
$\H(\l_1,...,\l_n)$ does not contain the origin $0 \in \C$. In
this case one has that every leaf  is transverse to every sphere around the origin $\n0 \in \C^n$
and has $\n0$ in its closure. Therefore, topologically, the leaves
are all cylinders over their intersection with a small sphere
$\s_\e$ around $\n0$, {\it i.e.}, they are of the form $(\L \cap
\s_\e) \times \R$. Hence they are immersed copies of  either $\C$
or $\C^*$. This is used in \cite{CKP} to prove the structural stability of this type of foliations.

On the other hand, if  $0$ is in $\H(\l_1,...,\l_n)$ we say that $F$
is in the {\it Siegel domain}. In this case if  the eigenvalues are
``generic enough'' then by \cite {CKP}  there is
an open dense set $\widehat M$ in $\C^n$ of points that belong to
a {\it Siegel leaf}. These leaves are embedded copies of $\C$
which have a unique point of minimal distance to $\n0$.

The space $M^*$ of Siegel leaves is a real analytic submanifold of
$\C^n$ of real codimension 2  defined by $M^* = M
\setminus\{\n0\}$, where $M$ is given by the equations:
$$ \langle F(\nz),\nz \rangle := \sum_{i=1}^n \l_i |z_i|^2 = 0 \,.$$
The topology and geometry of this type of manifolds has been investigated
by Lopez de Medrano, Verjovsky and Meersseman, obtaining
remarkable results (see for instance \cite{MV} and the bibliography on it).  
Following \cite{Me}, consider now $m$ linear vector fields in
$\C^n$, $n > 2m$, $F_j = (\l_1^j z_1,...,\l_n^j z_n)$,
$j=1,...,m$. For each $k = 1,...,n$, set $\Lambda_k =
(\l_k^1,...,\l_k^m)$, so we have $n$ vectors in $\C^m$. Assume
these vector   satisfy:

a) The {\it Siegel condition}: $\n0 \in
\H(\Lambda_1,...,\Lambda_n)$. That is, the origin $\n0 \in \C^m$ is
in the convex hull of the $n$ points in $\C^m$ determined by the
$\Lambda_i$; and

b) The {\it weak hyperbolicity condition}: no hyperplane through
any $2m$ of these $n$ points contains $\n0$.

Meersseman shows that the corresponding linear action of
$\C^m$ on $\C^n$ has an open dense set of {\it Siegel leaves}:
copies of $\C^m$ embedded with a unique point of minimal distance.
These points of minimal distance parameterize the space of Siegel
leaves, which is a smooth real submanifold of $\C^n$ of real
codimension $2m$, with a rich geometry and topology.
\end{examples}

\medskip
These are all examples of foliations that carry the Morse structure of the  function on $\C^n$ given by $(z_1,\cdots,z_n) \buildrel{Q}\over{\mapsto} |z_1|^2 + \cdots |z_n|^2$. 
\medskip

Recall that every point in $M$ is a critical point of the restriction 
$g_\L$ of $g$ to the corresponding leaf.

\begin{definition}  
Assume that a point  $\nz \in M^* := M \setminus \{\n0\}$ is a non degenerate contact. Then
define its {\it Morse index relative to $\F$} to be the
Morse index at $\nz$ of the restriction  of $g$ to the leaf   that contains $\nz$.
\end{definition}

Since the leaves have real dimension $2d$, {\it a priori} the
possible Morse indices go from $0$ to $2d$. However, not all these
indices are actually possible, because the fact that the foliation
is holomorphic imposes certain restrictions. In fact 
just as in \cite {Milnor:singular} or in \cite{Andreotti-Frankel}, it can be shown that 
there are no critical points  in the leaves of index greater than $d$.

\begin{examples}\label{pham-brieskorn}
{\bf i)}  Given integers $p,q > 2$,   let $F$
be the vector field defined by
$$F(z_1,z_2) = (qz_2^{q-1}, pz_1^{p-1})\,.$$
Its integral lines are the level curves of the  map
$(z_1,z_2) \buildrel{f} \over {\mapsto} (z_1^p - z_2^q)$.
In these examples the above  function $Q$  is
 Morse  restricted to each leaf. There is
one separatrix, $V:= \{z_1^p = z_2^q\}$, which is transversal to all
the spheres around $\n0$, and all other fibres meet the $z_1$-axis at
$p$ points and they have $q$  points on the $z_2$-axis.  A straightforward computation shows that these are all points with Morse index $0$. There is a third component of the polar variety $M$, which is a real analytic surface with an isolated singularity at $\n0$:
 $$M^-:= \{(z_1,z_2) \in \C^2 \,  {\big |} \, qz_2^{q-1} \oz_1 = - pz_1^{p-1} \oz_2 \,;\,z_1,z_2
\ne 0\}\,.$$ 
Each leaf has   $pq$
saddles in this surface. As a matter of fact, one can easily check that $M^-$ is ambient homeomorphic to the separatrix $V$. That is, there is a homeomorphism of $\C^2$ carrying $M^-$ into $V$. We do not know whether this is a coincidence or a special case of a general theorem.

	\medskip 
	\noindent
{\bf ii)}  Consider now  the vector field
$F(z_1,z_2) = (qz_2^{q-1}, 2z_1)\,, q > 2.$
Asides from the
separatrix $V = \{z_1^2 = z_2^q\}$, each leaf sufficiently near
$V$ has 2 points of minimal distance in the $z_1$-axis and $q$
saddles in the $z_2$-axis. When $\vert t \vert =  \left(\frac{2}{q} \right)^{\frac{q}{q-2}}$
the  
leaves  $f^{-1}(t)$ have degenerate contacts. This happens because a new component of $M$ appears, this is the  singular variety 
$$M^-= \{qz_2^{q-1} \bar z_1 + z_1
\bar z_2 = 0; z_1,z_2 \neq 0\}\,,$$ which does not pass through $\n0$, so we do not see it when we are near 
the origin. 
 As we consider leaves
which are further away from $V$, each leaf has $2$ minimal points
in the $z_1$-axis, as before, but they now have $q$ minimal points
in the $z_2$-axis and new singularities appear: each leaf has 2q
saddles along the surface $M^-$. Thence we see that even though all fibres $f^{-1}(t)$ are diffeomorphic for $t \ne 0$, yet, the way how they are embedded in $\C^2$ bifurcates as we pass from fibers inside the Milnor tube of radius 
$\left(\frac{2}{q} \right)^{\frac{q}{q-2}}$, to fibers outside this tube.

\medskip
\noindent
{\bf iii)}
For the linear vector field $(-z_2,z_1)$, whose integral
lines are the fibers of the map $z_1^2+ z_2^2$, 
the polar variety is a set of real dimension three
and each leaf $\mathcal{L}= f^{-1}(t)$ has a curve  of critical points, all of them 
degenerated contacts with a sphere. 
Yet,
 in \cite{Ito-Scardua} is proved that every linear foliation can be arbitrarily approximated by  linear foliations having only non-degenerate contacts (see  Section \ref{s: codim 1} below).

We remark that the same type of observations done above for  foliations defined by a polynomial $f =  z_1^2 - z_2^q$, $q > 2$, apply also in higher dimensions to all foliations defined by Pham-Brieskorn polynomials of the form:
$$f(z_1,\cdots,z_n) =  z_1^2 +  z_2^{a_2} + \cdots +  z_n^{a_n}\,, \quad  a_i > 2\,.$$
This was first observed in \cite{Ito-Scardua2}  for certain cases.

\end{examples}

\section{Generic contacts and smoothness of the polar variety}\label{section:teoremas}

 Our first result is the analogous in our setting of a classical result for polar varieties (compare with \cite{Le, LT}):
   
   \begin{theorem}\label{Theorem 1}
 Let $\mathcal{F}$ and  $g:\U \to \R$ be as above, $\U \subset \C^n$.  Let $  M$ be the polar variety of $\F$ relative to $g$ and $M^*:=  M \setminus \{\n0\}$. Then 
 $\mathcal{F}$ carries the Morse structure of $g$ if and only if  $ M^*$ is a smooth, reduced
 submanifold of $\mathbb{C}^n$ of codimension $2d$ and $M^*$ is everywhere transversal to $\mathcal{F}$. 
   \end{theorem}

The following lemma proves one side of Theorem \ref{Theorem 1}:

\begin{lemma}\label{M-subvariedad}
 If $\mathcal{F}$ carries the Morse structure of $g$, then $M^*$ is a smooth, reduced, real
 submanifold of $\mathbb{C}^n$ of codimension $2d$ and $M^*$ is transversal to $\mathcal{F}$  everywhere.
\end{lemma}

\begin{proof}
Let $\nz \neq \n0$ in $\U$ and consider  a foliated chart  $ U'$  around $\nz$  
with holomorphic coordinates  $(z_1, z_2, \dots, z_n)$,  such that the leaves  of $\F$ are locally given by:
$$ \big\{z_{d+1} = c_{d+1}, \dots, 
 z_{n} = c_n \big\}  , $$
where the $c_i$ are constants.
Thus we can consider $g_{\mathcal{L}_\nz}$ in $U'$ as a function of $(z_1,  \dots, z_d)$, or equivalently
of the real coordinates $(z_1, \overline{z}_1, \dots, z_d, \overline{z}_d)$.
Then the Hessian of  $g_{\mathcal{L}_\nz}$ at $\nz$ looks like
\begin{displaymath}
H (\nz):= Hess (g_{\mathcal{L}_\nz})(\nz)= \left(\begin{array}{ccccc}
\frac{\partial^2g_{\mathcal{L}_\nz} }{\partial z_1\partial z_1} 
& \frac{\partial^2 g_{\mathcal{L}_\nz}}{\partial \overline{z}_1 \partial z_1} &
\dots &  \frac{\partial^2g_{\mathcal{L}_\nz} }{\partial z_d \partial z_1} 
& \frac{\partial^2g_{\mathcal{L}_\nz} }{\partial \overline{z}_d \partial z_1}\\
\frac{\partial^2g_{\mathcal{L}_\nz}}{\partial z_1 \partial \overline{z}_1 } 
& \frac{\partial^2g_{\mathcal{L}_\nz} }{\partial \overline{z}_1 \partial \overline{z}_1 }&
\dots &  \frac{\partial^2g_{\mathcal{L}_\nz} }{\partial z_d \partial \overline{z}_1 } 
& \frac{\partial^2g_{\mathcal{L}_\nz} }{\partial \overline{z}_d \partial \overline{z}_1 }\\
\vdots &  \vdots && \vdots &  \vdots \\
\frac{\partial^2g_{\mathcal{L}_\nz} }{\partial z_1 \partial z_d} 
& \frac{\partial^2g_{\mathcal{L}_\nz} }{\partial \overline{z}_1\partial z_d }&
\dots &  \frac{\partial^2g_{\mathcal{L}_\nz} }{\partial z_d \partial z_d} 
& \frac{\partial^2g_{\mathcal{L}_\nz} }{\partial \overline{z}_d \partial z_d}\\
\frac{\partial^2g_{\mathcal{L}_\nz} }{\partial z_1\partial \overline{z}_d }
& \frac{\partial^2g_{\mathcal{L}_\nz} }{\partial \overline{z}_1 \partial \overline{z}_d} &
\dots &  \frac{\partial^2g_{\mathcal{L}_\nz} }{\partial z_d \partial \overline{z}_d} 
& \frac{\partial^2g_{\mathcal{L}_\nz} }{\partial \overline{z}_d \partial \overline{z}_d}
                                 \end{array} \right)
(\nz)
\end{displaymath}
Define $G_{j,0}(\nz):= \frac{\partial}{\partial z_j}g_{\mathcal{L}_\nz}(\nz)$ and
 $G_{0,j}(\nz):= \frac{\partial}{\partial \overline{z}_j}g_{\mathcal{L}_\nz}(\nz)$  for $1 \leq j \leq d$. 
Set 
$$G(\nz)= (G_{1,0}(\nz), G_{0,1}(\nz), \dots 
G_{d,0}(\nz), G_{0,d}(\nz))\,.$$
Then  $M$ is locally the zero set of the functions 
  $\{G_{j,0}(\nz),G_{0,j}(\nz)\}_{1 \leq j \leq d}$ and
 their jacobian  is
\begin{displaymath}
D_{\nz} G = \left(\begin{array}{cccccccc}
\frac{\partial G_{1,0}}{\partial z_1} & \frac{\partial G_{1,0}}{\partial \overline{z}_1}&
\dots &  \frac{\partial G_{1,0}}{\partial z_d} & \frac{\partial G_{1,0}}{\partial \overline{z}_d}& 
\frac{\partial G_{1,0}}{\partial z_{d+1}} & \dots & \frac{\partial G_{1,0}}{\partial \overline{z}_{n}}\\
\frac{\partial G_{0,1}}{\partial z_1} & \frac{\partial G_{0,1}}{\partial \overline{z}_1}&
\dots &  \frac{\partial G_{0,1}}{\partial z_d} & \frac{\partial G_{0,1}}{\partial \overline{z}_d}&
 \frac{\partial G_{0,1}}{\partial z_{d+1}} & \dots & \frac{\partial G_{0,1}}{\partial \overline{z}_{n}}\\
\vdots &  \vdots && \vdots &  \vdots &\vdots & & \vdots\\
\frac{\partial G_{d,0}}{\partial z_1} & \frac{\partial G_{d,0}}{\partial \overline{z}_1}&
\dots &  \frac{\partial G_{d,0}}{\partial z_d} & \frac{\partial G_{d,0}}{\partial \overline{z}_d}& 
\frac{\partial G_{d,0}}{\partial z_{d+1}} & \dots & \frac{\partial G_{d,0}}{\partial \overline{z}_{n}}\\
\frac{\partial G_{0,d}}{\partial z_1} & \frac{\partial G_{0,d}}{\partial \overline{z}_1}&
\dots &  \frac{\partial G_{0,d}}{\partial z_d} & \frac{\partial G_{0,d}}{\partial \overline{z}_d}& 
\frac{\partial G_{0,d}}{\partial z_{d+1}} & \dots & \frac{\partial G_{0,d}}{\partial \overline{z}_{n}}
                                 \end{array} \right)
\end{displaymath}
Observe that the submatrix formed by the first $2d$ columns of $D_{\nz} G (\nz)$
is  the matrix  $H (\nz)$. Since  $\mathcal{F}$ carries the Morse structure of $g$, 
 $D_{\nz} G (\nz)$ has maximal rank $2d$; therefore $M^*$ is defined locally as the inverse image of a regular value of  
the functions $\{G_{j,0}(\nz),G_{0,j}(\nz)\}_{1 \leq j \leq d}$. Hence it is
a reduced real  submanifold of codimension $2d$.

On the other hand, 
consider the vectors $v_j= (0,\dots , 0 , 1, 0, \dots, 0)$ 
which correspond to the canonical basis of $\mathbb{R}^{2n}$ with respect to  the coordinate system $(z_1, \overline{z}_1
,\dots ,z_n, \overline{z}_n )$. The set $\{v_j\}_{1\leq j\leq 2d}$ forms a basis of the space
 $T_{\nz} \mathcal{L}_\nz$ for all $\nz$ in $M^* \cap U'$. 
If we apply  $ D_{\nz}G $ to the vector $v_j$  we  see it  equals   the $j$th column of the matrix 
$ D_{\nz}G $. Since  $D_{\nz}G$ has rank $2d$ we have that
the tangent space $T_{\nz}\mathcal{L}_\nz$ is a direct sum complement of the kernel of $D_{\nz}G$.
Moreover the intersection between $T_{\nz}\mathcal{L}_\nz$ and $ker (D_{\nz}G)$ is trivial. Hence 
$T_{\nz}\mathcal{L}_\nz$ and $M^*$ have a transversal intersection.
\end{proof}

\medskip

The lemma below completes the proof of Theorem \ref{Theorem 1}:

\begin{lemma}\label{M-subvariedad=Morse}
 Let $M^*$ be a real   submanifold of $\mathbb{C}^n$  of codimension $2d$, reduced, and assume $M^*$ is 
transversal to $\mathcal{F}$ everywhere. Then $\mathcal{F}$ carries the Morse structure of $g$.
\end{lemma}
\begin{proof}

Let us take $(z_1, \overline{z}_1,\dots,z_n, \overline{z}_n )$ and  $G$ as in the proof of the previous lemma.
Recall $M = G^{-1}(0)$.
By hypothesis,  $M^*$ is a  reduced, real submanifold of dimension $2(n-d)$ in $\R^{2n}$, hence $D_{\nz} G$
 has rank $2d$.

Assume $\nv_0$ in  $\R^{2d}$ is different from zero and define $\nw_0 := (\nv_0, \n0)$ in $\R^{2n}$ for an appropriate coordinate
 chart as above. Then, as before,  $\nw_0$   is an element of $T_{\nz}\mathcal{L}_\nz$
and the matrices $D_{\nz}G$ and  $H(\nz)$ are related by $ D_{\nz}G = \left( H(\nz) , A\right)$, where $A$ is 
a real matrix of $2d \times 2(n-d)$; thus, $ D_{\nz}G (\nw_0) = H(\nz)(\nv_0)$.
By hypothesis,  the spaces $\L_\nz$ and  $M^*$ have a transversal intersection. Thus  the vector $\nw_0$ is not in 
the kernel of $D_{\nz}G$,  for dimensional reasons.
 Hence $H(\nz)(\nv_0)$ is different from zero. 
Since this happens for all $\nv_0$ not equal to 
zero in  $\R^{2d}$, we have that  $H(\nz)$ is non singular. Therefore $g_{\L_\nz}$ is non-degenerate.
\end{proof}

\begin{remark} 
Usually, when we speak of ``submanifolds" in differential topology we mean ``reduced submanifolds". Yet, the example below shows that this can be misleading and one must take care of the algebra behind the screen.
This is why  we added the word ``reduced" in the statement of Theorem \ref{Theorem 1}.

Consider the foliation on $\R^2$ given by the fibers of the map $(x,y) \mapsto x^4 + y$. Its leaves look like parabolas, but not quite. Consider now the contacts of this foliation with the foliation given by the lines $\{y = \rm constant\}$. Then the polar variety  is 
$$M= \{(x,y) \in \R^2 \, \big | \, x^3 = 0 \} \,.$$
So  this is just the $y$-axis, as a set,  which  of course is a smooth submanifold of $\R^2$. Yet, in this case the $y$-axis comes with multiplicity 3; it is the preimage of the critical value $0$ of the map $x^3$ in $\R^2$,  and the contacts are all degenerate.
\end{remark}

\begin{remark}\label{r: Haefliger} In \cite{Haefliger} Haefliger noticed that if $X$ is a complex manifold, 
of dimension say $n$,  $ \xi$ is a holomorphic regular foliation ({\it i.e.}, with no singularities) of dimension
$d$, $1 \le d < n$, and $N$ is a smooth differentiable submanifold of $X$ of real dimension $2n-2d$, which is everywhere 
transversal to $\xi$, then a foliated atlas for $\xi$ determines a canonical complex structure for $M$. This comes from the 
fact that a holomorphic foliated atlas for $\xi$ can be given by a locally finite collection $(U_\alpha, \Sigma_\alpha, p_\alpha)$ 
satisfying certain compatibility conditions, where the $U_\alpha$ form an open covering, the $\Sigma_\alpha$ are discs of dimension $2n-2d$ in $X$, transversal to $\xi$, 
and the $h_\alpha$ are local submersions $U_\alpha \to \Sigma_\alpha$ whose fibres give the plaques of $\xi$. Since the manifold $N$ is transversal to $\xi$,
 we can take the discs $\Sigma_\alpha$ as forming an atlas for $N$, and therefore the transversally holomorphic structure of $\xi$ determines a holomorphic structure for $N$.
Hence if $\F$ is as above and it carries the Morse structure of $g$, then the polar variety $M^*$ inherits from $\F$
a canonical complex structure.
\end{remark}

\begin{remark} Notice that the previous discussion holds with minor modifications 
for Morse functions on $\U$ with arbitrary Morse index at $\n0$. We have restricted the discussion 
to functions with Morse index $0$ because their level surfaces are spheres that bound a compact region 
around $\n0$, and we are interested in studying the topological and geometric behaviour of  holomorphic 
foliations near $\n0$.
\end{remark}

\begin{example}\label{ej:campo-Pham}
 
Consider the vector field $F(\nz) = (\lambda_1 z_{\sigma(1)}^{a_{\sigma(1)}},
\dots, \lambda_n z_{\sigma(n)}^{a_{\sigma(n)}})$, where $\sigma$ is a permutation of $\{1, \dots, n\}$ and $a_k\geq 2$ for all $k$.
Let $\F$ be the foliation defined by  $F$ and $Q(\nz)= |z_1|^2 + |z_2|^2 + \dots + |z_n|^2$. The corresponding
  polar variety $M= M(\F, Q)$ is defined by the function:
  $$f(z_1,\cdots,z_n) = \lambda_1 z_{\sigma(1)}^{a_{\sigma(1)}} \, \bar z_1 + \cdots +  \lambda_n z_{\sigma(n)}^{a_{\sigma(n)}} \, \bar z_n \,. $$
 In \cite{Seade}   is proved that this function has $\n0$ as its only critical point, so  
  $M$ is a real submanifold of $\mathbb{R}^{2n}$ of codimension
2. These singularities are also studied in \cite{RSV, Oka3}, where it is proved that if the permutation $\sigma$ is the identity, then the function $f$ is smoothly equivalent to a Pham-Brieskorn singularity. Hence in \cite{Seade2} these singularities are called {\it twisted Pham-Brieskorn singularities}.

\medskip

Now we study the transversality  between $\F$ and $M$ in special cases.

\medskip
\noindent {\bf Case  i) } Let $n=2$ and $F(\nz)= (\lambda_1 z_1^{a_1}, \lambda_2 z_2^{a_2})$. 
Then  $$M=\{(z_1, z_2) \in \mathbb{C}^2 |
h(\nz):= \lambda_1 z_1^{a_1} \oz_1 + \lambda_2 z_2^{a_2}\oz_2 =0\}\,,$$ so $M$ is defined as the zero set  of  the real functions 
$\psi(\nz):= Re (h(\nz))$ and $\phi(\nz):= Im (h(\nz))$. Notice that a  vector $v(\nz)$ belongs to $T_\nz M$ if and only if
$v(\nz)$ is orthogonal to the gradient of both $\psi(\nz)$ and $\phi(\nz)$. On the other hand 
$\{F(\nz), i F(\nz)\}$ is a basis of the tangent space of the leaf of $\F$ containing $\nz$.

Set  $\overline \nabla \psi(\nz):= \left(\frac{\partial \psi}{\partial \oz_1}, \dots, \frac{\partial \psi}{\partial \oz_n}\right)$ 
and consider   the gradient vector field   of $\psi(\nz)$.   
One has:  
\begin{eqnarray}\label{eq:producto}
 2 \left< F(\nz), \overline \nabla \psi(\nz) \right>_\mathbb{C} = \left< F(\nz), grad \, \psi(\nz) \right>_\mathbb{R}
- i \left<i F(\nz), grad \, \psi(\nz) \right>_\mathbb{R} \,.
\end{eqnarray}

Let us set:
$$A = |\lambda_1|^2 |z_1|^{2 a_1} + |\lambda_2|^2 |z_2|^{2a_2}\quad, \quad 
B =a_1 \lambda_1^2 z_1^{2 a_1 -1} \oz_1 +  a_2 \lambda_2^2 z_2^{2 a_2 -1} \oz_2 \,.$$
Then, we have:
$$\left< F(\nz), grad \, \psi \right>_\mathbb{R} = 2\,A + 2\,Re \,B \quad, \quad \left< iF(\nz), grad \, 
\psi \right>_{\mathbb{R}}= 2\,Im \,  B\, , $$
$$\left< F(\nz), grad \, \phi \right>_{\mathbb{R}}=- 2\,Im \,  B\quad \hbox{and} \quad
 \left< iF(\nz), grad \,\phi \right>_{\mathbb{R}}= -2\,A + 2\,Re \,  B \,.$$
 Let $\theta_k$ and $\theta_{\lambda_k}$ be the arguments of $z_k$ and $\lambda_k$ respectively.
The equations for $M$ can be written as:
$$|\lambda_1||z_1|^{a_1+1} = |\lambda_2||z_2|^{a_2 +1} \quad \hbox{and} \quad
e^{i \theta_{\lambda_1}}e^{i(a_1-1)\theta_1}= -e^{i\theta_{\lambda_2}} e^{i(a_2-1)\theta_2}\,.$$ Then:
\begin{eqnarray*}
B&=& a_1 |\lambda_1|^2|z_1|^{2 a_1} e^{i 2\theta_{\lambda_1}}e^{i(2a_1 -2)\theta_1} + a_2 |\lambda_2|^2 |z_2|^{2a_2}
e^{i2 \theta_{\lambda_2}} e^{i(2a_2 -2)\theta_2}\\
&=& (a_1 |\lambda_1|^2|z_1|^{2 a_1} +a_2 |\lambda_2|^2 |z_2|^{2a_2}) e^{i2\theta_{\lambda_2}} e^{i(2a_2 -2)\theta_2} �.
\end{eqnarray*}
Assume now that $F(\nz)$ is orthogonal to $grad\, \phi$. Then
$$B= \pm (a_1 |\lambda_1|^2|z_1|^{2 a_1} +a_2 |\lambda_2|^2 |z_2|^{2a_2})\,.$$ Hence 
$$\left< F(\nz), grad  \,  \psi \right>_\mathbb{R} = 2(1 \pm a_1) |\lambda_1|^2|z_1|^{2 a_1}  + 2(1 \pm a_2)|\lambda_2|^2 |z_2|^{2a_2}\,.$$
Thus  $F(\nz)$ is orthogonal to $grad \, \psi$ only at the origin. 
If we develop the equations for $iF(\nz)$ the results are similar. We thus get 
 that $\F$ is transversal to $M$ in every point away from the origin.

\medskip
\noindent {\bf Case  ii) } 
Now let  $F(\nz) = (\lambda_1 z_2^{a_2}, \lambda_2 z_1^{a_1})\,.$ The polar  variety is defined by:
$$\lambda_1 z_2^{a_2}\oz_1 + \lambda_2 z_1^{a_1}\oz_2=0\,.$$  
Arguing as above we get  that if  
we assume $F(\nz)$ is orthogonal to $grad  \,  \phi$, then:
 $$\left< F(\nz), grad  \,  \psi \right>_\R = (1 \pm a_1)|\lambda_1|^2|z_2|^{2a_2} + 
(1 \pm a_2)|\lambda_2|^2 |z_1|^{2a_1} \,.$$
This product is zero only if $\nz= \n0$. Similar computations work also for 
 $iF(\nz)$. We get that
 $\F$ and $M$ are transversal away from  the origin.

\medskip
\noindent {\bf Case  iii) } 
Now consider the following vector field in $\C^4$: $$F(\nz)= (\lambda_1 z_2^{2},
\lambda_2 z_3^{2}, \lambda_3 z_4^{2}, \lambda_4 z_1^{2})\,.$$
The equation of the polar  variety is:
\begin{equation}\label{eq. M}
h(\nz) := \lambda_1 z_2^{2} \oz_1 +\lambda_2 z_3^{2}\oz_2 +\lambda_3 z_4^{2} \oz_3 + \lambda_4 z_1^{2} \oz_4=0 \,.\end{equation}
We claim that for appropriate $\lambda_i$,  the foliation of $F$ is not transversal to $M^*$ at points of the form $ (z_1, z_2, z_3,0)$. To see this,
 assume that $F(\nz)$ is orthogonal to $grad  \,  \psi$ and $grad \, \phi$ at one of these points. We get:
\begin{eqnarray*}
0 = \frac{1}{2}\left<  F(\nz), grad  \,  \psi \right>_\R=  |\lambda_1|^2 |z_2|^{4} + |\lambda_2|^2|z_3|^{4} + |\lambda_4|^2 |z_1|^{4} 
\pm 2  |\lambda_1|^2|z_2|^{2} |z_1|^2 \\
= |\lambda_1|^2 |z_2|^{4} + |\lambda_4|^2 |z_1|^{4} +
(1\pm 2 ) |\lambda_1|^2|z_2|^{2} |z_1|^2 \,. \quad \quad  \, \,
\end{eqnarray*}  
 Notice that there are two cases to consider, depending to the $\pm$ sign on the last term in this equation.
 It is clear that with the $+$ sign this equation has no non-trivial solutions. Yet,   
it is an exercise to show that the equation with the $-$ sign has non-trivial solutions whenever 
$|\lambda_1|^2 - 4 |\lambda_4|^2>0$. Therefore in these cases  $\F$ is not transversal to $M^*$
at  points of the form $(z_1,z_2,z_3,0)$.

\end{example}

\section{Contact-analytic foliations}\label{s: contact-analytic}

From now on we assume  the Morse function $g$ is analytic on $U$.

\begin{definition}\label{contacto-analitico}
The foliation $\mathcal{F}$ is {\it contact-analytic} at $\n0$ 
with respect to  $g$  if the germ at $\n0$ of  the space
$M = M(\mathcal{F},g)$   is real
analytic.
\end{definition}

\begin{examples}
i) If $\F$ is 1-dimensional, then it is defined by a holomorphic
vector field $F$, and for the usual round metric $M$ is defined by
the analytic equations $\left<F(\nz),\nz\right> = 0$. Thus it is contact-analytic
for this metric.

\medskip

ii) If $\F$   is given by a holomorphic  action of some complex Lie group of
dimension $d$ with a fixed point at $\n0$ and trivial isotropy away from $\n0$, then  $M$ is contact-analytic
for this metric.

\medskip

iii) If $\F$ has codimension 1 and $n > 2$, then by Malgrange's theorem \cite{Malgrange} it has a first
integral. So we can assume it is defined by a holomorphic map $f$. 
The tangent space of $\F$ at each point $\nz$ is the kernel of $df$. 
 We  will see in Section \ref{s: codim 1} that $\F$ is contact-analytic for the usual metric.
\end{examples}

\medskip  

If the germ of $M$ at $\n0$ is real analytic, then it has finitely
many irreducible components. By Hironaka's work in  \cite{Hironaka:SubAnal}, 
 we can always  equip a small ball around $\n0 \in \C^n $
 with a Whitney stratification  compatible with $M$, and we can further assume  that 
 $\{\n0\}$ itself is a stratum and every other 
 stratum of $M^*$  has $\{\n0\}$ in its closure. Then
Verdier's Bertini-Sard theorem \cite{Ver} implies that 
sufficiently near the origin, 
 $g$    has no critical points on $M^* = M \setminus \{\n0\}$  in the stratified sense 
(compare with \cite{Milnor:singular}).  Then one gets that  every 
 stratum of $M^*$ meets transversally each sufficiently small sphere around $\{\n0\}$, and  each irreducible component of $M$ at $\n0$ is locally a cone
over its intersection with a sphere small enough around
$\n0$ (local conical structure of analytic sets).  
 
 If we assume that $\mathcal{F}$  carries the Morse structure of $g$, then $M^*$ is non-singular and we can assume that, after removing $\n0$, each irreducible component of $M$ at $\n0$ becomes a Whitney stratum. In this case one has that
    the gradient  flow $\mathcal G_{M^*}$  of $g|_{M^*}$ is transversal to all the $g$-spheres, its orbits accumulate at $\n0$ and it equips $M$ with the structure of a cone over $M_\e := M \cap \s_\e$ with vertex at $\n0$.
 
\begin{theorem}\label{Theorem 2}
Let $\mathcal{F}$ be   contact-analytic for the Morse function $g$ and assume it carries the Morse structure of $g$.
 Let $M_1,...,M_r$ be the  irreducible components  of the polar variety  $M$. 
Then:

i) Each 
$M_i^*:= M_i \setminus \{\n0\}$ is a real submanifold of $\C^n$ of codimension $2d$, transversal to $\mathcal{F}$ 
everywhere and transversal to all the $g$-spheres; 

ii) The $M_i$'s are pairwise disjoint away from $\n0$ and each $M_i^*$ consists of points with the same Morse index.
\end{theorem}

\begin{proof}
Item (ii) is the only statement that needs a proof, the other statement is  an immediate consequence of Theorem
 \ref{Theorem 1}. For proving (ii), supose
we have points $p_1$ and $p_2$ in some $M_i$ with different Morse index, and consider an 
arc $\alpha$ in $M_i$  joining these points. Then necessarily exits a point $q$ in this arc which 
 is a degenerate contact, because Morse singularities are stable and we cannot have a continuous family of Morse singularities   with distinct Morse indices. \end{proof}

\section{Morse theory and the topology of the leaves}\label{Sec: Morse} 

Recall from Section \ref{sec. definitions} that 
given a leaf  $\L$ of $\F$, we denote by $g_\L$   the
restriction of $g$ to  $\L$.  Lemma \ref{puntos-crit-tangentes} says that
 the critical points of $g_{\mathcal{L}}$ are the points where $\mathcal{L}$ is tangent to the level surfaces
of $g$. 

The {\it $\alpha$-limit} of a point $\nz \in \U$ under the flow $\mathcal G_{\F}$, is the set $\alpha(\nz)$ consisting of 
all points $\ny \in \U$ for which we can find a sequence $t_n \in \R$,  such that $ \lim_{t_n  \to - \infty} \mathcal G_\L(t_n,\nz) = \ny$. 
The {\it $\alpha$-limit} of a saturated set  is the union of the $\alpha$-limits of points in that set.

\medskip

The proof of the following two lemmas is an exercise, and it is exactly as the proof of propositions 3.1 and 3.2 in \cite{GSV} for 1-dimensional foliations. Notice that in the case of the usual metric, 
the gradient flow we consider here is the opposite of the radial flow envisaged in that article, so we consider $\alpha$-limits instead of $\omega$-limits.

\begin{lemma} If $M$ carries a Morse structure, then
the  {\it $\alpha$-limit} of $\mathcal G_{\F}$ is the polar variety $M$.
\end{lemma}

\begin{lemma} 
Let  $\widehat M$ be the saturated of $M$ by $\F$, {\it i.e.},  the union   of all leaves that meet $M$, and let  $K := \U_\e \setminus \widehat M$. 
Then the $\alpha$-limit of $K$ by the flow $\mathcal G_{\F}$ is the origin $\n0$, and the $\alpha$-limit of each  
$\nz \in \widehat M$  is the origin or the  point where the corresponding flow line meets $M$.
\end{lemma} 

We now have:

\begin{theorem}\label{t: Morse theory}  
Assume $\F$ is contact-analytic for $g$ and  carries  its  Morse structure. We now restrict
 $\F$ to a small closed ball $\B_\e$ centred at $\n0 \in \C^n$.
Let  $M$ be the intersection with $\B_\e$ of the corresponding polar variety in $\U$,  let  $\widehat M$ be the saturated of $M$ by $\F$ in $\B_\e$ and
 set  $K := \B_\e \setminus \widehat M$. Then:
\begin{enumerate}
\item The gradient flow endows
$K$ with the structure of a foliated cone with deleted vertex at $\n0$. In fact,  each leaf $\L \subset K$ is transversal to 
all the $g$-spheres around $\n0$, it is diffeomorphic to the product $(\L \cap \s_\e) \times \R$ and it has the origin in its closure. 

\item If $\L$ is compact, then its Euler-Poincar\'e characteristic $\chi(\L)$ equals the number of intersection points in $\L \cap M$, counted with sign. The sign is negative when the corresponding Morse index is odd, and 
positive   otherwise.
\end{enumerate}
\end{theorem}

\begin{proof}
Item (1) follows immediately from the two lemmas above. 
 Item (2) follows from the considerations above and the theorem of Poincar\'e-Hopf for manifolds with boundary, since 
the vector field corresponding to the gradient flow on the leaves is transversal to the boundary. Hence the Euler-Poincar\'e
 characteristic $\chi(\L)$ equals the total Poincar\'e-Hopf index of this flow on the leaf. At each singularity of the flow, the  
local Poincar\'e-Hopf index is determined by the corresponding Morse index: it is $1$ when the Morse index is  even (or $0$), and
 it is $-1$ when the Morse index is odd.
\end{proof}

\begin{remark}
 Notice that if a leaf $\L \subset \widehat M$ is compact, then not only its  Euler-Poincar\'e characteristic is determined as   above, but actually  classical Morse theory  tells us that  its whole topology  is  determined by its intersection with the boundary sphere and  the Morse indices at the points where $\L$ meets $M$. To some extent, this comments  apply  also for non-compact leaves, but here the situation can be far more complicated.
\end{remark}


\section{The foliation on the spheres}\label{s: spheres}

Notice that given the germ $(\F,\n0)$ and a small $g$-sphere
$ \s_{\e}$, the intersection of the leaves of $\F$ with
$ \s_{\e}$ defines a real analytic foliation $\F_{\e,g}$ on the
sphere, which is singular at $M_\e(\F,g) := \s_{\e} \cap M(\F,g) $.
 Away from $M_\e(\F,g)$ the leaves of $\F_{\e,g}$ have
real dimension $2d-1$.

Let us assume $\F$ is contact-analytic and carries the Morse structure of  $g$.
Let $M_1,...,M_r$ be the irreducible components of the germ
$(M,\n0)$. For $\e >0$ sufficiently small, each of these components
meets $ \s_{\e}$ transversally in a smooth submanifold of the
sphere, called {\it the link} $L_{i,\e}$  of the corresponding singularity.
It is well known,  by work of  Milnor and others (see  for instance \cite[Theorem 1.15]{CSS})  that there exists a
homeomorphism of the form:
$$\big(\B_\e, \B_\e \cap (\cup  M_i)\big) \longrightarrow
{\rm Cone}\big(\s_\e, \cup L_{i,\e} \big)\,.$$
On the other hand, for all $ \e > \e' > 0$  small enough we have a
diffeomorphism sending $\s_\e$ into $\s_\e'$ and the singular set 
of   $\F_{\e,g}$ into the corresponding one for $\F_{\e',g}$, {\it i.e.},
$(\s_\e, \bigcup L_{i,\e} \big) \cong (\s_\e',\bigcup L_{i,\e'} \big)$. Thus, it is natural to ask whether the 
corresponding foliations  are topologically conjugate. That is:

\nn {\bf Question 1.} Is there   a homeomorphism (or a
diffeomorphism)  from $\s_\e$ into $\s_\e'$  carrying the leaves of $\F_{\e,g}$ into those of
$\F_{\e',g}$?  In other words, are these foliations equivalent,
either topologically or differentiably?

\medskip 
The answer in general seems to be  negative, though it is positive in some settings, as shown by 
 the theorem below. This result is
an extension of the equivalent theorem in \cite{GSV} for 1-dimensional foliations.

\begin{theorem}\label{Theorem 3} 
If the Morse index of each point in $M^*$ is $0$, then for all $\e
> \e' >0$  small enough, the induced foliations on the
spheres are smoothly equivalent. That is, there is a
diffeomorphism
$$(\s_\e, \F_{\e}, M_\e \big) \buildrel{\cong} \over {\longrightarrow} (\s_{\e'},\F_{\e'}, M_{\e'}
 \big)\,,$$ taking  leaves of $\F_{\e}$ into
leaves of $\F_{\e'}$, where $M_\e$ and  $M_{\e'}$ are the singular sets of $\F_{\e}$
and $\F_{\e'}$, respectively.
\end{theorem}

\begin{proof}
All points in $M^*$ have Morse index 0, so they are points of local minimal $g$-distance to $\n0$ in their leaves. Hence every leaf of $\mathcal{F}_\e$ near $M$ is compact and diffeomorphic to a sphere of dimension $2d-1$. 
Let $V_\e$ be a tubular neighbourhood of $M_{\e}$ in $\s_\e$  such that its boundary $\partial V_\e$
is formed by such leaves.

  Let  $\mathcal G_{M^*}$ 
be the gradient  flow of $g|_{M^*}$, as in 
Section \ref {s: contact-analytic}. Its orbits accumulate at $\n0$ and they are 
  transversal to all sufficiently small $g$-spheres. This gives  a diffeomorphism $M_\e  \rightarrow M{_\e'}$ as in \cite{Milnor:singular}, which is our starting point for 
  comparing the foliation $\F_\e$ with the foliation $\F_{\e'}$ for $\e > \e' > 0$. The claim (to be proved) is that these foliations are topologically, and actually smoothly, equivalent.  By compactness, we can assume that $ \e' $ is close enough to $\e$ so that every leaf in $\partial V_\e$ meets $\s_{\e'}$ transversally, for otherwise we can take a finite number of spheres with $\e = \e_0 > \e_1 > \cdots > \e_r = \e'$ and apply the following arguments step by step.

Consider the gradient  flow   $\mathcal G_{M^*}$, and for each $\nz \in M_\e$,   let  $\gamma_t(\nz)$ be the flow line  of  $\mathcal G_{M^*}$ passing through $\nz$.
Notice that $\gamma_t(\nz)$ intersects  every  sphere $\s_{\e_t}$, $\e \ge \e_t \ge \e' $, in a single point. We   call $\ny(\nz)$
 the corresponding point in $\s_{\e'}$. 
Each point in  $\gamma_t(\nz)$ between the points $\nz$ and  $\ny(\nz)$ belongs to a leaf 
$\mathcal{L}$ in $\mathcal{F}$ such that the intersection  $\mathcal{L}\cap S_\e$ is a sphere contained in the interior of the tubular
 neighbourhood $V_\e$. Also, each leaf $\L_\partial$ in $\partial V_\e$ corresponds to a point  $\np(\L_\partial)$  in $M$ whose $g$-distance to $\n0$ is strictly less than $\e'$, which is the point of contact of $\L_\partial$ with a $g$-sphere. 
 
 For each leaf $\L_\partial$ in $\partial V_\e$, let $\np(\L_\partial)$ be its point of contact  in $M$, as above. 
Let $\gamma_t(\np(\L_\partial))$ be the flow line through $\np(\L_\partial)$, let $\nz(\L_\partial)$  be the point where this arc meets $\s_\e$ and  
 $\ny(\L_\partial)$  the point where $\gamma_t(\np(\L_\partial))$  meets $\s_{\e'}$. The points in the arc $[\np(\L_\partial), \nz(\L_\partial)\big)$ in the curve 
 $\gamma_t(\np(\L_\partial))$ determine a 1-parameter family of $(2d-1)$ spheres in $\s_\e$, nested around the point 
 $\nz(\L_\partial)$, forming a disc $\mathbb{D}^{\mathcal{N}}_{\e, \nz}$ of dimension $2d$, transversal to $M^*$ at $\nz(\L_\partial)$. Notice that the interval $[\np(\L_\partial), \nz(\L_\partial)\big)$  is taken to be open on the right because the end-point corresponds to the centre of the disc.
 Doing so for all points in $M_\e$ we get back the neighbourhood $V_\e$, naturally identified with the normal bundle of $M_\e$ in the sphere, and each normal fibre $\mathbb{D}^{\mathcal{N}}_{\e, \nz}$ is foliated by concentric spheres which are leaves of $\F_\e$.

 Similarly, the points in the arc $[\np(\L_\partial), \ny(\L_\partial)\big)$ in the curve 
 $\gamma_t(\np(\L_\partial))$ determine a 1-parameter family of $(2d-1)$ spheres in $\s_{\e'}$, nestled around the point 
 $\ny(\L_\partial)$, forming a disc  $\mathbb{D}^{\mathcal{N}}_{\e', \ny}$ of dimension $2d$, transversal to $M^*$ at $\ny(\L_\partial)$. Doing so for all points in $M_{\e'}$ we get a tubular neighbourhood $V_{\e'}$ of $M_{\e'}$ in $\s_{\e'}$.

Let $\mathcal{C}_{\e,\e'}$ be the cylinder  bounded by the spheres $\s_\e$ and $\s_{\e'}$. This cylinder is a union of two compact foliated sets $\mathcal{A}_{\e,\e'}$, $\mathcal{V}_{\e,\e'}$ obtained as follows:
If $\buildrel{\circ} \over {V}_\e$ is the interior of $V_\e$, then $\mathcal{A}_{\e,\e'}$ is the subset of $\mathcal{C}_{\e,\e'}$ obtained by 
 saturating $\s_\e \setminus \buildrel{\circ} \over {V}_\e$ by $\F$. This is a cylinder foliated by the leaves of $\F$ intersected with $\mathcal{C}_{\e,\e'}$. Each leaf in $\mathcal{A}_{\e,\e'}$ is transversal to all the $g$-spheres of radius $r$ with  $\e \ge r \ge \e'$. The set $\mathcal{V}_{\e,\e'}$ is the compact neighbourhood of $M \cap \mathcal{C}_{\e,\e'}$
 in $\mathcal{C}_{\e,\e'}$, obtained by saturating $V_\e$ by $\F$. Notice that the intersection of these two sets $\mathcal{C}_{\e,\e'}$, $\mathcal{V}_{\e,\e'}$ is exactly the subset obtained by saturating $\partial V_\e$.

Let us prove now  that the foliations $\F_\e$ and $\F_{\e'}$ are smoothly equivalent.  We do it in pieces, first on 
 $\mathcal{A}_{\e,\e'}$. Let  $\mathcal G_{\F}$ be the gradient flow on the leaves in $\mathcal{A}_{\e,\e'}$, defined in Section \ref{Sec: Morse}. This is a smooth flow on all of 
 $\mathcal{A}_{\e,\e'}$, with no singularities, and its flow lines are transversal to all the spheres. This determines a diffeomorphism 
 $$H: (\s_\e \setminus  \buildrel{\circ} \over {V}_\e) \longrightarrow  (\s_{\e'} \setminus  \buildrel{\circ} \over {V_{\e'}})  $$
that carries leaves of $\F_\e$ into leaves of $\F_{\e'}$. 

Notice that $H$ defines a diffeomorphism  from $\partial  V_{\e}$ into $\partial V_{\e'}$. 
We now  extend $H$  to the interior of   these neighbourhoods. Recall that a diffeomorphism is already given between $M_\e$ and $M_{\e'}$ 
and it is defined by the gradient flow $\mathcal G_{M^*}$. For the moment, denote this map  by $h$.  This is  compatible with $H$ in the sense that if $V_{\e}$ and  $V_{\e'}$ are identified with the unit disc normal bundles of $M_\e$ and $M_{\e'}$ in the corresponding spheres, then for each $\nz \in M_\e$, $H$ is a bundle map between the corresponding sphere bundles. That is, for each $\nz \in M_\e$, $H$ carries the  boundary of the normal disc $\mathbb{D}^{\mathcal{N}}_{\e, \nz}$ into the boundary of the normal disc $\mathbb{D}^{\mathcal{N}}_{\e', h(\nz)}$. Thus we denote $h$ also by $H$.

It remains to extend $H$ to the interior of the normal discs $\mathbb{D}^{\mathcal{N}}_{\e, \nz}$, $\mathbb{D}^{\mathcal{N}}_{\e', H(\nz)}$, 
and we have already defined it on $M^*$ in a compatible way. Recall that each such disc is foliated by spheres and these are 
parameterized by an arc $[\np(\L_\partial), \nz(\L_\partial)\big)$, contained in an integral line of the gradient flow on $M^*$, where $\L_\partial$ is the boundary of the disc $\mathbb{D}^{\mathcal{N}}_{\e, \nz}$, 
$\np(\L_\partial)$ is the point where the leaf of $\F$ that determines $\L_\partial$ meets $M$, and $\nz = \nz(\L_\partial)$ 
is the point in $M_\e$ whose normal fibre is determined by $\L_\partial$. Similarly, the disc $\mathbb{D}^{\mathcal{N}}_{\e', H(\nz)}$ is 
parameterized by the arc $[\np(\L_\partial), H(\nz)\big) \subset [\np(\L_\partial), \nz\big)$. 

Choose a smooth family of diffeomorphisms, parameterized by  $M_\e$, that carry each interval  $[\np(\L_\partial), \nz\big)$ 
into the corresponding interval 
$[\np(\L_\partial), H(\nz)\big)$.  This determines which sphere in each  $\mathbb{D}^{\mathcal{N}}_{\e, \nz}$ goes to which sphere 
in  $\mathbb{D}^{\mathcal{N}}_{\e', H(\nz)}$. The rest is now an exercise, using that the diffeomorphism is already given 
on the boundary of each disc.  For this we construct  integrable vector fields on $V_\e$ and $V_{\e'}$, which are tangent
 to each disc $\mathbb{D}^{\mathcal{N}}_{\e, \nz}$, $\mathbb{D}^{\mathcal{N}}_{\e', H(\nz)}$,  transversal to each leaf (sphere) 
in these discs, and singular at the  center of the corresponding disc. These vector fields can be regarded as being 
families parameterized by the points in $M_\e$ and $M_{\e'}$ respectively, of radial vector fields on each normal disc 
$\mathbb{D}^{\mathcal{N}}_{\e, \nz}$, $\mathbb{D}^{\mathcal{N}}_{\e', H(\nz)}$. These vector fields give us a natural 
way for extending the given diffeomorphisms 
$$\partial \mathbb{D}^{\mathcal{N}}_{\e, \nz} \to \partial  \mathbb{D}^{\mathcal{N}}_{\e', H(\nz)}$$
to the interior of the discs, carrying leaves of $\F_e$ into leaves of $\F_{\e'}$. 

We thus get an extension of $H$ to a diffeomorphism $(\s_\e, M_\e) \rightarrow (\s_{\e'}, M{_\e'}),$ taking  leaves of $\F_{\e}$ into
leaves of $\F_{\e'}$,

\end{proof}

\section{Codimension one foliations }\label{s: codim 1}

We now restrict the discussion to
codimension one foliations. We assume further that $\F$ has a 
  first integral near $\n0$, so there exists  a holomorphic map-germ  $(\C^n, \n0) \buildrel{f}\over{\rightarrow} (\C,0) $  such that 
 near $\n0$, the leaves of $\F$ are the fibers of $f$. By  
 Malgrange's theorem in \cite{Malgrange}, this hypothesis is a restriction only when $n=2$.  
 
 Our first task is to find a suitable set of generators for the tangent bundle $T\F$ of $\F$ away from $\n0$ (to include 
the origin we must speak of the tangent sheaf). For instance, if $n=2$, then the holomorphic vector field $F= (\frac{\partial f}{\partial z_2}, - \frac{\partial f}{\partial z_1})$ satisfies $df(F(\nz)) = 0$ for all $\nz \in \U$. 
Since $T\F$ is 1-dimensional, we have that $F$ determines  a basis for $T\F$ at all points in $\U \setminus \{\n0\}$. 
 
 For $n>2$ one has that the tangent sheaf of $\F$ may not be locally free. Thence, generally speaking, we need more than $n-1$ germs of holomorphic vectors fields to generate 
  $T\F$.  The proof of the following lemma is an exercise: 
 
 \begin{lemma}\label{generadores}
Consider the vector fields $$v_{j,k}(\nz):= \left(0, \dots,0 , \frac{\partial f}{\partial z_k}
(\nz),0, \dots,0,  -\frac{\partial f}{\partial z_j}(\nz),0, \dots,0 \right)\,,$$
where the entries  $\frac{\partial f}{\partial z_k}(\nz)$ and $ -\frac{\partial f}{\partial z_j}(\nz)$
are  in the $j$th and $k$th position respectively. Then  $\{v_{j,k}\}_{1\leq j< k\leq n}$
is a set of generators of each tangent space $T_{\nz }\L_\nz$, for all $\nz \in \U \setminus \{\n0\}$.
\end{lemma}

 Of course this lemma is giving us too many generators, which can be unsatisfactory in some sense. Yet, these vector fields are so simple and canonical, that they lead to nice expressions for the polar variety of  $M$. This will be useful in the sequel.

As before, let  $M(\F,g)$  be  the polar variety  of $\F$ with respect to a Morse function $g$ with Morse 
index $0$ at $\n0$. We have:

\begin{proposition}\label{M-var-analitica} The variety  $ M(\F,g)$ is defined by the equations 
  \begin{equation}\label{equations M}
   \frac{\partial g}{\partial z_k} \frac{\partial f}{\partial z_j} =
\frac{\partial g}{\partial z_j} \frac{\partial f}{\partial z_k} \;,
\end{equation}
for all  $j,k $ in $\{1, \dots ,n\}$. Hence, if $g$ is analytic then  $\F$ is contact-analytic with respect to $g$.
 In particular, for  $Q(z_1,...,z_n) = |z_1|^2 +... +  |z_n|^2\;$ we obtain: $$M(\F,Q)  = 
\left\{\nz \in \C^n \left| \;  \oz_k \frac{\partial f}{\partial z_j}(\nz)  \, = \,
 \oz_j\frac{\partial f}{\partial z_k}(\nz) \; , \; 1\leq j<k\leq n\right.\right\} \;.$$
\end{proposition}

\begin{proof}
Notice that the last two statements in \ref{M-var-analitica}  follow from the equations (\ref{equations M}), so we only need to show that these equations actually define $M$.
Since $f$ has an isolated critical point at $\n0$, given a point $\nz$ in  
$\C^n\backslash \{\n0\}$ one has  $\frac{\partial f}{\partial z_j}(\nz)\neq 0$ for some $j$. Let  $\L$ be the leaf of $\F$
containing $\nz$. 
 Then the implicit function theorem says that there is a neigborhood $U'$ of $\nz$  
where, on the leaf $\L$, the coordinate function  $z_j$ depends holomorphically on the other variables.
Deriving  implicitly the equation $f(\nz)= c$  we find:
$$ \frac{\partial z_j}{\partial z_k} = 
-\frac{ \frac{\partial f}{\partial z_k}}{\frac{\partial f}{\partial z_j}}$$
at all points in $U'\cap \L$. Using this to find the critical points of $g_\L$ we obtain the desired equations by a straight-forward computation.
\end{proof}

We know from 
 the examples in  \ref{pham-brieskorn} that there are   foliations that have 
degenerate contact points for the function $Q(x_1,y_1,x_2,y_2)=  x_1^2 +  y_1^2 + x_2^2 +  y_2^2$, as for instance the 
foliation defined by  the polynomial map $f(z_1, z_2)= z_1^2 + z_2^2$.  
Yet, in these example one can easily destroy the degenerate contacts. 
For instance:

\begin{examples}
 Assume $Q_\Lambda(x_1,y_1,x_2,y_2)= a_1 x_1^2 + b_1 y_1^2 + a_2x_2^2 + b_2 y_2^2$ where $a_1, b_1, a_2, b_2$ are
positive real numbers, and consider the foliation $\mathcal{F}$ defined by
the function $f(z_1,z_2)= z_1^2+ z_2^2$. The polar  variety is the set defined by $$M= \left\{(z_1, z_2)\, \in \,\mathbb{C}^2 \,| \;
z_2[a_1 Re \, z_1
- i b_1 Im \,  z_1] = z_1[a_2 Re \,  z_2 - i b_2 Im \,  z_2] \right\}\,.$$ If we have  $a_1> a_2$ and $b_1 > b_2$, then the polar  variety is formed only by the two  coordinate
axis, all contacts are non degenerate  and
each leaf has two saddle points on the  $z_1$-axe and two minimal points on the $z_2$-axe. Analogous statements hold for $a_1< a_2$ and $b_1 < b_2$.
\end{examples}

Notice that in this example the foliation $\F$ is determined by the linear form $df = 2z_1dz_1 + 2z_2dz_2$. 
In  \cite{Ito-Scardua},   T. Ito and B. Scardua  prove:

  \medskip
  \noindent
{\bf Theorem} [Ito-Scardua].\label{ito-scardua}
 {\it Every linear foliation $\F$ on $\C^n$ can be arbitrarily approximated by 
foliations that carry  the Morse structure of $Q$.}
  \medskip
  
  Let us consider now homogeneous foliations of degree more than 1.  This means, by Malgrange's theorem, that the foliation
  $\H_k$ is defined by the differential of some homogeneous polynomial of degree more than 2.
  The simplest case is:

\begin{theorem}\label{Fermat}
 Let $H^n_k$ be the  Fermat   polynomial 
 $\lambda_1 z_1^k + \lambda_2 z_2^k + \dots +\lambda_n z_n^k$.
 If $k>2$, then the  foliation $\H_k$ defined by the fibers of $H_k^n$ has a  Morse structure compatible with
the quadratic form $Q= |z_1|^2 +   |z_2|^2 + \cdots + |z_n|^2 $, and the corresponding polar variety consists of a finite number of complex lines through the origin.
 \end{theorem}
 
That $\H_k$  has a  Morse structure compatible with
the quadratic form $Q$ is an immediate consequence of Theorem 1 and the  two propositions below.  Notice that the claim that $M$ is a union of complex lines was  proved in \cite{Ito-Scardua} for linear foliations

\begin{proposition}\label{p:homogeneous-smooth}
Let $H^n_k=\lambda_1 z_1^k + \lambda_2 z_2^k + \dots +\lambda_n z_n^k$.
 If $k>2$ then away from $\n0$, the polar  variety $M= M(\mathcal{H}_k, Q)$ is a differentiable reduced submanifold of $\mathbb{R}^{2n}$, of dimension 2.
\end{proposition}
\begin{proof}
 
For simplicity  assume $z_1 \neq 0$. In this case, the set of vectors 
\begin{eqnarray*}
  v_2 &=& (\lambda_2 z_2^{k-1}, -\lambda_1 z_1^{k-1}, 0, 0,\dots, 0)\\
v_3 &=& (\lambda_3 z_3^{k-1},0, -\lambda_1 z_1^{k-1}, 0, \dots, 0)\\
\vdots && \vdots \\
v_n &=& (\lambda_n z_n^{k-1},0, 0, \dots, 0, -\lambda_1 z_1^{k-1})
\end{eqnarray*}
forms a basis of $T_\nz\L_\nz$. Then the polar  variety is defined by the equations
\begin{eqnarray}\label{ec-contacto-homogeneo}
 G_j(\nz):= \lambda_1 z_1^{k-1} \oz_j- \lambda_j z_j^{k-1} \oz_1 =0
\end{eqnarray}
 for all  $j= 2, \dots, n$. In other words, 
 $M$  is the zero set of the real equations:
$$\psi_j (\uz) := 2Re\, G_j(\uz) \quad \hbox{and} \quad  \phi_j(\uz):= 2 Im \,G_j(\uz)$$ for all $j= 2, \dots, n$. 
We set  $\Phi:=(\psi_2, \phi_2, \dots, \psi_n, \phi_n)$.
Notice one has: 
\begin{displaymath}
  \frac{\partial \psi_j}{\partial z_l} = \left\{ \begin{array}{ccc}
                      \lambda_1 (k-1) z_1^{k-2} \oz_j  - \overline{\lambda}_j \oz_j^{k-1} & &     l=1\\
               - \lambda_j (k-1) z_j^{k-2} \oz_1 + \overline{\lambda}_1 \oz_1^{k-1}& & l = j\\
		  0 &  & l\neq j,1
                                                   \end{array}
\right.
\end{displaymath}
\begin{displaymath}
  \frac{\partial \psi_j}{\partial \oz_l} = \left\{ \begin{array}{ccc}
                      \overline{\lambda}_1(k-1) \oz_1^{k-2} z_j  - \lambda_j z_j^{k-1} & &     l=1\\
               - \overline{\lambda}_j (k-1) \oz_j^{k-2} z_1 + \lambda_1 z_1^{k-1}& ¡ & l = j\\
		  0 &  & l\neq j,1
                                                   \end{array}
\right.
\end{displaymath}
\begin{displaymath}
  \frac{\partial \phi_j}{\partial z_l} = \left\{ \begin{array}{ccc}
                     i[ \lambda_1 (k-1) z_1^{k-2} \oz_j  + \overline{\lambda}_j \oz_j^{k-1}] & &     l=1\\
                i[-\lambda_j (k-1) z_j^{k-2} \oz_1 - \overline{\lambda}_1 \oz_1^{k-1}]& & l = j\\
		  0 &  & l\neq j,1
                                                   \end{array}
\right.
\end{displaymath}
\begin{displaymath}
  \frac{\partial \phi_j}{\partial \oz_l} = \left\{ \begin{array}{ccc}
                     i[- \overline{\lambda}_1(k-1) \oz_1^{k-2} z_j  - \lambda_j z_j^{k-1}] & &     l=1\\
               i[ \overline{\lambda}_j (k-1) \oz_j^{k-2} z_1 + \lambda_1 z_1^{k-1}]&  & l = j\\
		  0 &  & l\neq j,1
                                                   \end{array}
\right.
\end{displaymath}
Set  $a_j(\nz):= \lambda_1 (k-1) z_1^{k-2} \oz_j $, $b_j(\nz):= \lambda_j z_j^{k-1} $,
$c_j(\nz) := \lambda_j (k-1) z_j^{k-2} \oz_1$ and $d(\nz):= \lambda_1 z_1^{k-1} $. Then the jacobian matrix
 $D \Phi$ is (we omit writing the variable $\nz$ for simplicity):
\begin{displaymath}
\left(\begin{array}{cccccccc}
a_2 - \overline{b}_2 & -b_2 + \overline{a}_2 & -c_2 + \overline{d} & d - \overline{c}_2  & \dots & 0 & 0 \\
i[a_2 + \overline{b}_2] & i[-b_2 - \overline{a}_2] & i[-c_2 - \overline{d}] & i[d + \overline{c}_2] & \dots & 0 & 0 \\
a_3 - \overline{b}_3 & -b_3 + \overline{a}_3 & 0 & 0& \dots & 0 & 0 \\
i[a_3 + \overline{b}_3] & i[-b_3 - \overline{a}_3] & 0 &0  & \dots & 0 & 0 \\
\vdots & \vdots&   \vdots & \vdots &\ddots  & \vdots & \vdots \\
a_n - \overline{b}_n & -b_n + \overline{a}_n & 0 & 0  & \dots & -c_n + \overline{d} & d - \overline{c}_n   \\
i[a_n + \overline{b}_n] & i[-b_n - \overline{a}_n] & 0 & 0  & \dots & i[-c_n - \overline{d}] &i[ d + \overline{c}_n] \\
               \end{array}
 \right)
\end{displaymath}
To prove that  $D\Phi$ has rank $2n-2$ consider its submatrix $B_n$ obtained by eliminating the first two columns
of $D\Phi$.   Then Proposition \ref{p:homogeneous-smooth} follows from the 
 following lemma.
 \end{proof}

\begin{lemma}
 Let $B_n$ be as above. Then: 
 $$det (B_n(\nz)) \, = \,  (2i)^{n-1}\prod_{j=2}^{n}[|d(\nz)|^2 - |c_j(\nz)|^2] \,=\, (2i)^{n-1}\rho (\nz) [2k-k^2]^{m}\,,
 $$
for some positive integers $\rho (\nz)$, $m$.  Hence the matrix $B_n(\nz)$ is invertible when $k>2$.
\end{lemma}

\begin{proof}
 The proof is by induction.
Assume $n=2$, then the determinant of $B_2$ is
\begin{displaymath}
 det (B_2) = i\left|\begin{array}{cc}
-c_2 + \overline{d} &  d -  \overline{c}_2\\
- c_2 - \overline{d} & d + \overline{c}_2
                   \end{array}
 \right| 
= 2i |d|^2 - 2i |c_2|^2 \,.
\end{displaymath}
Now we assume that every matrix $B_{n-1}$  as above, in  $n-1$ variables,  satisfies 
$det (B_{n-1})= (2i)^{n-2}\prod_{j=2}^{n-1}[|d|^2 - |c_j|^2]$, and  we let $B_n$ be
\begin{displaymath}
 B_n =\left(\begin{array}{cccccccc}
 -c_2 + \overline{d} & d - \overline{c}_2 &0 &0 & \dots & 0 & 0 \\
 -ic_2 - i\overline{d} & id + i\overline{c}_2 &0&0& \dots & 0 & 0 \\
0 & 0& -c_3 + \overline{d} & d - \overline{c}_3  & \dots & 0 & 0 \\
 0 &0 & -ic_3 -i\overline{d} & id + i\overline{c}_3 & \dots & 0 & 0 \\
  \vdots  & \vdots& \vdots & \vdots &\ddots  & \vdots & \vdots \\
 0 & 0 & 0 & 0 & \dots & -c_n + \overline{d} & d - \overline{c}_n   \\
 0 &0   & 0 & 0  & \dots & -ic_n - i\overline{d} & id + i\overline{c}_n \\
               \end{array}
 \right)
\end{displaymath}
Using the last two columns to evaluate  the determinant we get:
$$ det(B_n)=2i[|d|^2 - |c_n|^2] det(B_{n-1})\,,$$
where
\begin{displaymath}
 B_{n-1} =\left(\begin{array}{cccccccc}
 -c_2 + \overline{d} & d - \overline{c}_2 &0 &0 & \dots & 0 & 0 \\
 -ic_2 - i\overline{d} &i d + i\overline{c}_2 &0&0& \dots & 0 & 0 \\
0 & 0& -c_3 + \overline{d} & d - \overline{c}_3  & \dots & 0 & 0 \\
 0 &0 & -ic_3 - i\overline{d} & id + i\overline{c}_3 & \dots & 0 & 0 \\
  \vdots  & \vdots& \vdots & \vdots &\ddots  & \vdots & \vdots \\
 0 & 0 & 0 & 0 & \dots & -c_{n-1} + \overline{d} & d - \overline{c}_{n-1}   \\
 0 &0   & 0 & 0  & \dots & -ic_{n-1} - i\overline{d} & id + i\overline{c}_{n-1} \\
               \end{array}
 \right)
\end{displaymath}
Then the induction hypothesis implies,
 $$det(B_n)=2i[|d|^2 - |c_n|^2](2i)^{n-2}\prod_{j=2}^{n-1}[|d|^2 - |c_j|^2]
= (2i)^{n-1}\prod_{j=2}^{n}[|d|^2 - |c_j|^2] \,,$$
proving the first statement in the lemma. 

Now choose the indices $\{1,\cdots,n\}$ appropriately, so that 
 $z_1, z_2, \dots,z_m \neq 0$ and $z_{m+1}= \dots= z_{n}=0$, for some $m \ge 1$. Notice that $c_j(\nz)= 0$ for all
 $j> m$, and $d(\nz) \neq 0$, so we have:
$$det(B_n)=(2i)^{n-1} |d|^{2(n-m)}\prod_{j=2}^{m}[|d|^2 - |c_j|^2] \,.$$
Recall the points of $M$ satisfy the equation \ref{ec-contacto-homogeneo}, that can be seen in norms as
 $|\lambda_1| |z_1|^{k-2} = |\lambda_j| | z_j|^{k-2}$ for all $j$. Hence:
\begin{eqnarray*}
|d|^2 - |c_j|^2 
&=&|\lambda_1|^2| z_1|^{2(k-1)}-(k-1)^2 | \lambda_1|^2  |z_1|^{2(k-2)} |z_1|^2\\
&=& |\lambda_1|^2| z_1|^{2(k-1)}(1 - (k-1)^2) \,.
\end{eqnarray*}
Therefore  the determinant of $B_n$  takes the form
\begin{eqnarray*}
 det( B_n)&=& (2i)^{n-1} (|\lambda_1| |z_1|^{k-1})^{2(n-m)}\prod_{j=2}^{m}|\lambda_1|^2| z_1|^{2(k-1)}(2k-k^2)\\
&=& (2i)^{n-1} |\lambda_1|^{2n}|z_1|^{2(k-1)n} [2k-k^2]^{m}
\end{eqnarray*}
completing the proof of  the lemma.
\end{proof}

\begin{proposition}\label{transversal}
 Let $H_k^n= \lambda_1 z_1^k + \lambda_2 z_2^k + \dots +\lambda_n z_n^k$ and $k>2$, then the intersection of 
$\mathcal{H}_k$ and $M^*$ is transversal everywhere.
\end{proposition}

\begin{proof}
Assume $z_1 \neq 0$ and let
 $\psi_j$, $\phi_j$ and $v_j$ be as in the proof of Proposition \ref{p:homogeneous-smooth}, for $j= 2, \dots, n$. 
 The vectors $\{v_2, iv_2, \dots, v_n, i v_n\}$ form a basis  of $T_\nz \L_\nz$ as a vector space over $\R$. 
 Notice  that  the leaf $\L_{\nz}$ is not transversal to $M^*$ at $\nz$ if and only if  
there exists a  vector  $v$ in $T_{\nz}\L_{\nz}$  orthogonal to $grad \, \phi_j$ and $grad \, \psi_j$ for all 
$j= 2, \dots, n$.

Let $  \overline \nabla $ be as in Section \ref{section:teoremas}, so for each $j$ one has 
$  \overline \nabla \psi_j(\nz) = (\frac{\partial \psi_j}{\partial \oz_1}, \dots, \frac{\partial \psi_j}{\partial \oz_n})\,.$
Then:
$$ 2 \left< F({\nz}),  \overline \nabla \psi({\nz}) \right>_\mathbb{C} = \left< F({\nz}), grad \, \psi({\nz}) \right>_\mathbb{R}
- i \left<i F({\nz}), grad \, \psi({\nz}) \right>_\mathbb{R}\,.$$

\noindent
Remember that $v_j$ belongs to  $T_\nz M^*$ if and only if  the products 
$\left< v_j, grad \,  \psi_l\right>_{\mathbb{R}} $ and  $ \left< v_j, grad \, \phi_l\right>_{\mathbb{R}}$ 
are zero for all  $l= 2, \dots, n$.
In particular we   consider the  case $l=j$. We set:
\begin{eqnarray*}
 A_j:= |\lambda_1|^2 |z_1|^{2k-2} + |\lambda_j|^2 |z_j|^{2k-2},\\   
B_j:=(k-1)[\lambda_1 \lambda_j z_j^{k-1} \oz_j z_1^{k-2}  +\lambda_1 \lambda_j z_1^{k-1} \oz_1 z_j^{k-2}] \,.
\end{eqnarray*}
Then:
$$\left< v_j, \overline \nabla \psi_j \right>_{\mathbb{C}} = -A_j +B_j \quad \hbox{and} \quad  \left< v_j, \overline \nabla \phi_j \right>_{\mathbb{C}} = iA_j+ iB_j\,.$$
Hence:
$$\left< v_j, grad \,  \psi_j\right>_{\mathbb{R}} = - 2\, A_j + 2\, Re \, B_j \quad , \quad 
\left< iv_j, grad \,  \psi_j\right>_{\mathbb{R}} = - \, 2\, Im \, B_j\,,$$
$$ \left< v_j, grad \,  \phi_j\right>_{\mathbb{R}} = -\, 2\, Im \, B_j \quad \hbox{and} \quad
 \left< iv_j, grad \,  \phi_j\right>_{\mathbb{R}} =2\, A_j + 2\, Re  \, B_j\,.$$
Now,  using the equation $|\lambda_1| |z_1|^{k-2} = |\lambda_j||z_j|^{k-2}$ 
and the polar decomposition of $B_j$ we get:
$$
 B_j =  (k-1) [ |\lambda_1| |\lambda_j| |z_j|^k |z_1|^{k-2}+  |\lambda_1| |\lambda_j| |z_1|^k |z_j|^{k-2}]
 e^{i (\theta_{\lambda_1}+ \theta_{\lambda_j})}
e^{i(k-2) (\theta_1+ \theta_j)}
$$
$$
  =  (k-1) [ |\lambda_j|^2 |z_j|^{2k- 2}+  |\lambda_1|^2  |z_1|^{2k-2}]
 e^{i (\theta_{\lambda_1}+ \theta_{\lambda_j})} e^{i(k-2) (\theta_1+ \theta_j)}    \; .  \quad \qquad $$

Assume $v_j$ is orthogonal to $grad \,  \phi_j$, {\it i.e.}, $-Im \,B_j=0$, then 
$$B_j=\pm (k-1) [ |\lambda_j|^2 |z_j|^{2k- 2}+  |\lambda_1|^2  |z_1|^{2k-2}] \,.$$ Therefore,
$$\left< v_j, grad \,  \psi_j \right>_\R=2\,( |\lambda_j|^2 |z_j|^{2k- 2}+  |\lambda_1|^2  |z_1|^{2k-2}) [-1 \pm (k-1)] \;.$$
Since  $z_1\neq 0$,  we must have  $\left< v_j, grad \,  \psi_j \right>_\R \neq 0$. Similar computations
work also for $iv_j(\nz)$. Hence we get that, for all $j$, neither 
$v_j(\nz)$ nor  $iv_j(\nz)$  are  in $T_\nz M$, 
 so we arrive to Proposition \ref{transversal}.
\end{proof}

To complete the proof of 
Theorem \ref{Fermat} we must show that the polar variety $M$ consists of a finite number of lines through the origin. 
From Proposition \ref{p:homogeneous-smooth} we know that $M$ has real dimension 2, it is smooth away from $\n0$ and 
it is analytic at $\n0$. The result will be proved if we show that $M$ is a union of complex lines. This is a 
consequence of the fact that $M$  is defined as the zero set of  polynomials of the form
$P_{j,l}(\nz):= \frac{\partial h_k}{\partial z_j} \oz_l - \frac{\partial h_k}{\partial z_l} \oz_j$, 
with   $\frac{\partial h_k}{\partial z_j}$  homogeneous (in the complex sense) 
of degree $k-1$. 
\qed

\medskip

Now consider an arbitrary homogeneous polynomial $h$
 of degree $k$ in $n$ complex variables, with a unique critical
point at  $\n0 \in \C^{n}$, and its corresponding foliation $\mathcal H$. 
Let  $\mathcal C$ be  the   space of coefficients of
homogeneous polynomials of
degree $k$  in $n$ complex variables. 
The general homogeneous
equation of degree $k$  in $n$  variables is:
$$\sum_{\alpha_1 + \cdots + \alpha_n = k} a_{\alpha_1, \cdots , \alpha_n}
z_1^{\alpha_1}...\,z_n^{\alpha_n}\,.$$  
   The polynomials
defining  hypersurfaces in $\C^n$  with non-isolated singularities correspond to a Zariski closed subset 
of $\mathcal C$.
 Hence its complement $\Omega$  is connected, and therefore every homogeneous polynomial 
 of degree $k$ in $n$ complex variables with a unique critical
point at  $\n0 \in \C^{n}$, is isotopic to a
  Fermat polynomial $H^n_k:= z^k_1 +...+z^k_n$.   Then there exists an isotopy $\mathcal I$ of $\C^n$, fixing the 
origin $\n0$, carrying the foliation $\mathcal H$ of  $h$ into the foliation $\H_k$ of $H^n_k$, which carries the 
Morse structure of the quadratic form $Q$, by Theorem \ref {Fermat}. We obtain

\begin{corollary}
Let  $\mathcal H$ be a  codimension 1, homogeneous  holomorphic foliation germ in $\C^n$. 
 Then $\mathcal H$ is isotopic to the foliation $\H_k$ of $H^n_k$, and therefore carries a Morse structure 
compatible with  a Morse function isotopic to 
the quadratic form $Q$.
\end{corollary}

The results in this section, together with the aforementioned theorem of Ito-Scardua and other known facts, suggest the
 following:

\medskip
\noindent
{\bf Conjecture.} Let $\F$ be a  holomorphic foliation in an open neighbourhood $\U$ of $\n0 \in \C^n$, singular at $\n0$.
 Then there is an open set $\mathcal G$ in the space of all Morse functions with a critical point at $\n0$ with 
Morse index $0$, such that for every $g \in \mathcal G$, the contacts of $\F$ with the fibers of $g$ in a 
neighbourhood of $\n0$ are all non-degenerate.

\medskip
{\it Acknowledgement}. The authors are grateful to professors Alberto Verjovsky, Xavier G\'omez-Mont, C\'esar Camacho, Bruno Scardua, Paulo Sad and Bernard Teissier, for very helpful and fruitful conversations.

\vspace{0.1cm}{\bf Beatriz Lim\'on and Jos\'e Seade:}

Instituto de Matem\'aticas, Unidad Cuernavaca, 

Universidad Nacional Aut\'onoma de M\'exico,

Av. Universidad s/n, Lomas de Chamilpa, C.P. 62210,

Cuernavaca, Morelos, M\'exico.

\medskip

{email: betlimon@matcuer.unam.mx, jseade@matcuer.unam.mx}

\end{document}